\documentclass[12pt]{amsart}
\usepackage[centering]{geometry}

\usepackage{amsmath,amsthm,amssymb}
\usepackage{times}
\usepackage{enumerate}

\theoremstyle{plain}
\newtheorem*{theorem*}{Theorem}
\newtheorem{theorem}{Theorem}[section]

\newtheorem{lemma}[theorem]{Lemma}

\newcommand{\be}{\begin{equation}}
\newcommand{\ee}{\end{equation}}

\newcommand{\lt}{\left}
\newcommand{\rt}{\right}
\newcommand{\goto}{\rightarrow}
\newcommand{\R}{\mathbb{R}}

\newcommand{\kp}{\kappa}
\newcommand{\vp}{\varphi}

\newcommand{\tC}{\tilde{C}}
\newcommand{\n}{\mathbf{n}}

\theoremstyle{definition}

\newtheorem{remark}[theorem]{Remark}

\numberwithin{equation}{section}
\numberwithin{equation}{section}

\begin{document}
\setlength{\baselineskip}{1.2\baselineskip}

\title[Neumann boundary value problem for gernera curvature flow]
{Neumann boundary value problem for gernera curvature flow with forcing term}

\author{Ling Xiao}
\address{Department of Mathematics, Rutgers University,
Piscataway, NJ 08854}
\email{lx70@math.rutgers.edu}

\begin{abstract}
In this paper, we prove long time existence and convergence results for a class of general curvature flows with Neumann boundary condition. 
This is the first result for the Neumann boundary problem of non Monge-Ampere type curvature equations. Our method also works for 
the corresponding elliptic setting.
\end{abstract}

\maketitle

\section{Introduction}
\label{int}
\setcounter{equation}{0}

This paper, we  consider the deformation of convex graphs over bounded, convex domains 
$\Omega\subset\R^n, n\geq 2,$ to convex graphs with prescribed general curvature and Neumann boundary condition.
More precisely, let  $\Sigma(t)=\{X:=(x, u(x, t))| (x, t)\in\Omega\times[0, T)\},$ we study the long time existence and convergence of the following flow problem
\be\label{int.0}
\lt\{\begin{aligned}
\dot{u}&=w\lt(f(\kp[\Sigma(t)])-\Phi(x, u)\rt)\,\,&\mbox{in $\Omega\times[0, T)$}\\
u_\nu&=\vp(x, u)\,\,&\mbox{on $\partial\Omega\times[0, T)$}\\
u|_{t=0}&=u_0\,\,&\mbox{in $\Omega$},
\end{aligned}
\right.
\ee
where $\Phi, \vp: \bar{\Omega}\times\R\goto\R$ are smooth functions, $\nu$ denotes the outer unit normal to $\partial\Omega,$
and $u_0: \bar{\Omega}\goto\R,$ is the initial value.
The flow equation in \eqref{int.0} is equivalent to say $X$ satisfies 
\[\dot{X}=(f(\kp[\Sigma(t)])-\Phi)\n,\]
where $\n$ is the upward unit normal of $\Sigma(t).$

We are goint to focus on the locally convex hypersurfaces. Accordingly, the function $f$ is assumed to be defined in the convex cone
$\Gamma_n^+\equiv\{\lambda\in\R^n: \mbox{each component $\lambda_i>0$} \}$ in $\R^n$ and satisfying the fundamental structure conditions:
\be\label{int.2}
f_i(\lambda)\equiv\frac{\partial f(\lambda)}{\partial\lambda_i}>0\,\,\mbox{in $\Gamma_n^+, 1\leq i\leq n,$}
\ee
and
\be\label{int.3}
\mbox{$f$ is a concave function.}
\ee
In addition, $f$ will be assumed to satisfy some more technical assumptions. These include
\be\label{int.4}
f>0\,\,\mbox{in $\Gamma_n^+, f=0$ on $\partial\Gamma_n^+,$  }
\ee
\be\label{int.5}
f(1, \cdots, 1)=1,
\ee
and 
\be\label{int.6}
\mbox{$f$ is homogeneous of degree one.}
\ee
Moreover, for any $C>0$ and every compact set $E\subset\Gamma^+_n,$ there is
$R=R(E, C)>0$ such that 
\be\label{int.7}
f(\lambda_1, \cdots, \lambda_{n-1}, \lambda_n+R)\geq C, \forall \lambda\in E.
\ee

An example of functions satisfies all assumptions above is given by $f=\frac{1}{2}\lt[H_n^\frac{1}{n}+(H_n/H_l)^{\frac{1}{n-l}}\rt],$
where $H_l$ is the normalized $l$-th elementary symmetric polynomial. However, we point out that the pure curvature quotient
$(H_n/H_l)^{\frac{1}{n-l}}$ does not satisfy \eqref{int.7}.

Since for a graph of $u,$ the induced metric and its inverse matrix are given by 
\be\label{ce2.12}
g_{ij}=\delta_{ij}+u_iu_j\,\,\mbox{and $g^{ij}=\delta_{ij}-\frac{u_iu_j}{w^2},$}
\ee
where $w=\sqrt{1+|Du|^2}.$
Following \cite{CNS4}, the principle curvature of graph $u$ are eigenvalues of the symmetric matrix $A[u]=[a_{ij}]:$
\be\label{ce2.13}
a_{ij}=\frac{\gamma^{ik}u_{kl}\gamma^{lj}}{w},\,\,\mbox{where $\gamma^{ik}=\delta_{ij}-\frac{u_iu_k}{w(1+w)}.$}
\ee
The inverse of $\gamma^{ij}$ is denoted by $\gamma_{ij},$ and
\be\label{ce2.14}
\gamma_{ij}=\delta_{ij}+\frac{u_iu_k}{1+w}.
\ee
Geometrically $[\gamma_{ij}]$ is the square root of the metric, i.e. $\gamma_{ik}\gamma_{kj}=g_{ij}.$
Now, for any positive definite symmetric matrix $A,$ we define the function $F$ by
\[F(A)=f(\lambda(A)),\]
where $\lambda(A)$ denotes the eigenvalues of $A.$
We will use the notation 
\[F^{ij}(A)=\frac{\partial F}{\partial a_{ij}},\,\,F^{ij, kl}=\frac{\partial^2 F}{\partial a_{ij}\partial a_{kl}}(A).\]
The matrix $[F^{ij}(A)]$ is symmetric and has eigencalues $f_1, \cdots, f_n,$ and by \eqref{int.2},
$[F^{ij}(A)]$ is positive definite. Moreover, by \eqref{int.3}, $F$ is a concave function of $A$, that is 
\[F^{ij, kl}(A)\xi_{ij}\xi_{kl}\leq 0,\]
for any $n\times n$ matrix $[\xi_{ij}].$
 
We rewrite equation \eqref{int.0} as following
 \be\label{int.1}
\lt\{\begin{aligned}
\dot{u}&=w\lt(F\lt(\frac{\gamma^{ik}u_{kl}\gamma^{lj}}{w}\rt)-\Phi(x, u)\rt)\,\,&\mbox{in $\Omega\times[0, T)$}\\
u_\nu&=\vp(x, u)\,\,&\mbox{on $\partial\Omega\times[0, T)$}\\
u|_{t=0}&=u_0\,\,&\mbox{in $\Omega$},
\end{aligned}
\right.
\ee
 
We will prove
\begin{theorem}
\label{inth.1}
Let $\Omega$ be a smooth bounded, strictly convex domain in $\R^n.$ Let $\Phi, \vp: \bar{\Omega}\times\R\goto\R,$ be smooth functions satisfy
\be\label{int.8}
\Phi>0 \,\,\mbox{and $\Phi_z\geq 0$,}
\ee 
\be\label{int.9}
\vp_z\leq c_\vp<0.
\ee
Let $u_0$ be a smooth, convex function that satisfies the compatibility condition on $\partial\Omega$:
\be\label{int.10}
\lt.\nu^iu_i-\vp(x, u)\rt\vert_ {t=0}=0.
\ee
Moreover, we assume 
\be\label{int.11}
f(\kp[\Sigma_0])-\Phi(x, u_0)\geq 0,
\ee
where $\Sigma_0=\{(x, u_0(x))|x\in\Omega\}.$
Then there exists a solution $u\in C^\infty(\bar{\Omega}\times(0, t))\cap C^{\alpha+2, 1+\alpha/2}(\bar{\Omega}\times[0, t))$ of equation \eqref{int.1} for all $t>0.$ As $t\goto\infty,$ the function $u(x, t)$ smoothly converges to a smooth limit function 
$u^\infty,$ such that $u^\infty$ satisfies the Neumann boundary value problem
\be\label{int.12}
\lt\{\begin{aligned}
F\lt(\frac{\gamma^{ik}u^\infty_{kl}\gamma^{lj}}{w}\rt)&=\Phi(x, u^\infty)\,\,&\mbox{in $\Omega$}\\
u^\infty_\nu&=\vp(x, u^\infty)\,\,&\mbox{on $\partial\Omega$,}\\
\end{aligned}
\rt.
\ee
where $\nu$ is the outer unit normal of $\partial\Omega.$
\end{theorem}
\bigskip

\begin{remark}\label{inrk.1}
The short time existence for equation \eqref{int.1} comes from Theorem 5.3 in \cite{LSU} and the implicit function theorem. 
\end{remark}

By applying short time existence theorem,  we know that the flow exists for $t\in [0, T^*),$ for some $T^*>0$ very small. In the following sections, we fix
   $T<T^*,$ and establish the uniform 
$C^2$ bounds for the solution $u$ of \eqref{int.1} in $(0, T].$ Since our estimates are independent of $T,$ repeating this process we obtain the longtime existence of equation \eqref{int.1}.

Neumann boundary problem has attracted lots of attetions through these years. In particular, the existence for equations of Monge-Ampere type was studied
in \cite{LTU} in the 80s'; later Jiang, Trudinger, and Xiang \cite{JTX} addapted and developed the methods in \cite{LTU} to a generalized Monge-Ampere type equation with Neumann boundary condition. Recently, Ma and Qiu proved the existence of solutions to $\sigma_k$ Hession equations with Neumann boundary condition in their beautiful paper \cite{MQ}, in this paper they solved a long lasting conjecture by Trudinger in 1986.   The Neumann boundary problems for  parabolic equation have been wildly studied too. For example,  mean curvature flow with Neumann boundary condition have been studied in \cite{ AW, BG, AS};  Guass curvature flow with Neumann boundary condition have been studied in \cite{SS}. 

Our paper is oganized as follows: In Section \ref{se} we prove the uniform estimate for $\dot{u},$ which also implies the convexity for $u(\cdot, t), t\in [0, T].$ This is used in Section \ref{ge} to derive the $C^0$ and $C^1$ estimates. Section \ref{ce} is the most important section, in which we derive the $C^2$ estimates for $u$. Finally, in Section \ref{cvg} we combine all results above to prove the convergence of solution of \eqref{int.1} as $t\goto\infty$.

\section{Speed estimate}
\label{se}
\setcounter{equation}{0}
\begin{lemma}
\label{selm.1}
As long as a smooth convex solution of \eqref{int.1} exists, we have
\be\label{se.1}
\min\{\min\limits_{t=0}\dot{u}, 0\}\leq\dot{u}\leq\max\{\max\limits_{t=0}\dot{u}, 0\}.
\ee
\end{lemma}
\begin{proof}
If $(\dot{u})^2$ achieves a positive local maximum at $(x, t)\in\partial\Omega\times[0, T]$ then at this point we would have
\be\label{se.2}
(\dot{u})^2_\nu=2\dot{u}\dot{u}_\nu=2(\dot{u})^2\vp_z<0,
\ee
which leads to a contradiction. Thus, we assume $(\dot{u})^2$ achieves maximum at an interior point.
Now let's denote
\[\tilde{G}(D^2u, Du, u)=wF\lt(\frac{\gamma^{ik}u_{kl}\gamma^{lj}}{w}\rt)-w\Phi(x, u)\]
and $r=(\dot{u})^2.$ Then, a straight forward calculation gives us
\be\label{se.3}
\dot{r}=\tilde{G}^{ij}r_{ij}-2\tilde{G}^{ij}\dot{u}_i\dot{u}_j+\tilde{G}^sr_s+2\tilde{G}_ur.
\ee
Since 
\be\label{se.4}
\tilde{G}_u:=\frac{\partial \tilde{G}}{\partial u}=-w\Phi_u\leq 0,
\ee
we have
\be\label{se.5}
\dot{r}-\tilde{G}^{ij}r_{ij}-\tilde{G}^sr_s\leq 0.
\ee
By the maximum principle we know that a positive local maximum of $(\dot{u})^2$ can not occur at an interior point of $\Omega\times(0, T].$
Therefore, we proved this Lemma.
\end{proof}

\begin{lemma}
\label{selm.2}
A solution of \eqref{int.1} satisfies $\dot{u}>0$ for $t>0$ if $0\not\equiv \dot{u}\geq 0$ for $t=0.$
\end{lemma}
\begin{proof}
Since 
\be\label{se.6}
\dot{u}=\tilde{G}(D^2u, Du, u),
\ee
differentiating it with respect to $t$  we get
\be\label{se.7}
\frac{d}{dt}u_t=\tilde{G}^{ij}(u_t)_{ij}+\tilde{G}^s(u_t)_s+\tilde{G}_uu_t.
\ee
Therefore, for any constant $\lambda$ we have
\be\label{se.8}
\frac{d}{dt}(u_te^{\lambda t})=\tilde{G}^{ij}(u_te^{\lambda t})_{ij}+\tilde{G}^s(u_te^{\lambda t})_s+\tilde{G}_u(u_te^{\lambda t})+\lambda u_te^{\lambda t}.
\ee
We fix $t_0>0$ and a constant $\lambda$ such that $\lambda+\tilde{G}_u>0$ for $(x, t)\in\bar{\Omega}\times[0, t_0].$
By the strong maximum principle we see that $u_te^{\lambda t}$ has to vanish identically if it vanishes in $\Omega\times(0, t_0),$ which leads to a contradiction.

If $u_te^{\lambda t}=0$ for $(x, t)\in\partial\Omega\times (0, t_0),$ then we would have
\be\label{se.9}
(u_te^{\lambda t})_\nu=\vp_z(u_te^{\lambda t})=0
\ee
contradicts the Hopf Lemma.
\end{proof}
\begin{remark}
\label{sermk.1}
Lemma \ref{selm.2} impies that, if we start from a strictly convex surface $\Sigma_0$ satisfies \eqref{int.11}, then as long as the flow exists,
the flow surfaces $\Sigma(t)$ are strictly convex and satisfies 
$f(\kp[\Sigma(t)])-\Phi(x, u)>0.$
\end{remark}

\bigskip
\section{$C^0$ and $C^1$ estimates}
\label{ge}
\setcounter{equation}{0}
The strict convexity of $u$ and the fact  that $\vp(\cdot, z)\goto -\infty$ uniformly as $z\goto\infty$ implies that $u$ is uniformly bounded from above.
By Lemma \ref{selm.2}
\be\label{ge.1}
u(x, t)=u(x, 0)+\int_0^t\dot{u}(x, \tau)d\tau\geq u(x, 0)
\ee
 we know $u$ is bounded from below as well. To conclude, we have
\begin{theorem}[$C^0$ estimates]
Under our assumption \eqref{int.11} on $u_0,$ a solution of equation \eqref{int.1} satisfies
\be\label{ge.2}
|u|\leq C_0,
\ee
where $C_0=C_0(u_0, \vp).$
\end{theorem}

\begin{theorem}[$C^1$ estimates]
For a convex solution $u$ of equation \eqref{int.1}, the gradient of $u$ remains bounded during the evolution,
\be\label{ge.3}
|Du|\leq C_1,
\ee
where $C_1=C_1(|u|_{C^0}, \Omega, \vp).$
\end{theorem}
\begin{proof}
The proof is the same as Theorem 2.2 in \cite{LTU}, for readers convenience we include it here.
By the convexity of $u$ we have for any $t\in [0, T]$
\be\label{ge.4}
\max\limits_{\Omega}|Du(\cdot, t)|=\max\limits_{\partial\Omega}|Du(\cdot, t)|.
\ee
Let $x_0\in\partial\Omega$ and let $\tau$ be a direction such that $\nu\cdot\tau=0$ at $x_0.$
Let $B=B_R(z)$ be an interior ball at $x_0,$ L be the line through $x_0$ in the direction of $-\nu,$
and $L$ intersects $\partial B$ at $y_0.$ Then $z=\frac{1}{2}(x_0+y_0),$  we also let $y$ be the unique point such that 
$\frac{y-z}{|y-z|}=\tau.$

Now let $\omega$ be an affine function such that $\omega(x_0)=u(x_0, t)$ and $D\omega=Du(x_0, t).$ Then $\omega\leq u(x, t),\,\,x\in\Omega$ and
\be\label{ge.5}
\begin{aligned}
\omega(z)&=\omega(x_0)+D\omega(x_0)\cdot(z-x_0)\\
&=u(x_0, t)+Du(x_0, t)\cdot\frac{z-x_0}{|z-x_0|}\cdot|z-x_0|\\
&\geq u(x_0, t)-M_1R,
\end{aligned}
\ee
where we assume $\vp(x, u)\leq M_1$ in $\bar{\Omega}\times[-C_0, C_0].$
Therefore,
\be\label{ge.6}
D_\tau u(x_0, t)=D_\tau\omega(x_0)=\frac{\omega(y)-\omega(z)}{|y-z|}\leq\frac{u(y, t)-u(x_0, t)+M_1R}{R}\leq\frac{2C_0}{R}+M_1.
\ee
Since $\tau,$ $x_0,$ and $t$ are arbitrary, we are done.
\end{proof}
\bigskip

\section{$C^2$ estimates}
\label{ce}
\setcounter{equation}{0}
First of all, we will list some evolution equations that will be used later. Since the calculations
are straightforward, we will only state our results here.
\begin{lemma}
\label{celm.1}
Let $u$ be a solution to the general curvature flow \eqref{int.1}. Then we have the following evolution equations:\\
(i)$\frac{d}{dt}g_{ij}=-2(F-\Phi)h_{ij},$\\
(ii)$\frac{d}{dt}\n=-g^{ij}(F-\Phi)_i\tau_j,$\\
(iii)$\frac{d}{dt}\n^{n+1}=-g^{ij}(F-\Phi)_iu_j,$\\
(vi)$\frac{d}{dt}h^j_i=(F-\Phi)^j_i+(F-\Phi)h^k_ih^j_k$,\\
where $g_{ij}, h_{ij}$ are the first and second fundamental forms, $\n$ is the upward unit normal to $\Sigma(t),$ $\n^{n+1}=\lt<\n, e^{n+1}\rt>,$
and $h_i^j=g^{jk}h_{kj}.$ 
\end{lemma}

\subsection{$C^2$ interior estimates}
\label{ce1}
\setcounter{equation}{0}
In this subsection, we will prove the following theorem.
\begin{theorem}
\label{ce1th.1}
Let $\Sigma(t)=\{(x, u(x, t))| x\in \Omega, t\in[0, T]\}$ be the flow surfaces, where $u(x, t)$ satisfies equation \eqref{int.1} and 
\[\n^{n+1}\geq 2a>0\,\,\mbox{on $\Sigma(t), \forall t\in [0, T].$}\]
For $X\in \Sigma(t),$ let $\kp_{\max}(X)$ be the largest principle curvature of $\Sigma(t)$ at $X.$ Then
\be\label{ce1.1'}
\max\limits_{\bar{\Omega}_T}\frac{\kp_{\max}}{\n^{n+1}-a}\leq C_2(\Phi, |u|_{C^1})\lt(1+\max\limits_{\partial\Omega_T}\kp_{\max}\rt),
\ee
where $\Omega_T=\Omega\times(0, T].$
\end{theorem}
\begin{proof}
Let's consider 
\[M_0=\max\limits_{\bar{\Omega}_T}\frac{\kp_\text{max}}{\n^{n+1}-a},\]
we assume $M_0>0$ is attained at an interior point $(x_0, t_0)\in\Omega\times(0, T].$
We can choose a coordinate such that $\kp_1=\kp_\text{max},$  $h^j_i=\kappa_i\delta_{ij},$ and $g_{ij}=\delta_{ij}$ at $(x_0, t_0).$
In the following, $h_{ij},$ $h^j_i$ means the same.

At $(x_0, t_0),$ $\psi=\frac{h_{11}}{\n^{n+1}-a}$ achieves its local maximum. Hence at this point we have
\be\label{ce1.1}
\frac{h_{11i}}{h_{11}}-\frac{\nabla_i\n^{n+1}}{\n^{n+1}-a}=0.
\ee
Moreover, by Lemma \ref{celm.1}
\be\label{ce1.2}
\begin{aligned}
\frac{\partial}{\partial t}\psi&=\frac{\dot{h_{11}}}{\n^{n+1}-a}-\frac{h_{11}\dot{\n}^{n+1}}{(\n^{n+1}-a)^2}\\
&=\frac{1}{\n^{n+1}-a}\lt\{\nabla_{11}F-\nabla_{11}\Phi+(F-\Phi)\kp_1^2\rt\}
+\frac{h_{11}}{(\n^{n+1}-a)^2}(F-\Phi)_iu_i.
\end{aligned}
\ee
Since
\be\label{ce1.3}
\nabla_{11}\Phi=\Phi_{x_1x_1}(x, u)+2\Phi_zu_1+\Phi_zu_{11},
\ee
\be\label{ce1.4}
\nabla_{11}u=\lt<X, e_{n+1}\rt>_{11}=\lt<h_{11}\n, e_{n+1}\rt>=h_{11}\n^{n+1},
\ee
and
\be\label{ce1.5}
\begin{aligned}
\nabla_{11}F&=F^{ij}h_{ij11}+F^{ij,rs}h_{ij1}h_{rs1}\\
&=F^{ij}(h_{11ij}-h^2_{11}h_{ij}+h_{ik}h_{kj}h_{11})+F^{ij, rs}h_{ij1}h_{rs1}.
\end{aligned}
\ee
Combine \eqref{ce1.2}-\eqref{ce1.5} we get at $(x_0, t_0)$
\be\label{ce1.6}
\begin{aligned}
&\frac{\partial}{\partial t}\psi-F^{ii}\nabla_{ii}\psi\\
&=\frac{1}{\n^{n+1}-a}\lt\{F^{ii}h_{ii11}+F^{ij, rs}h_{ij1}h_{rs1}-\nabla_{11}\Phi+(F-\Phi)\kp_1^2\rt\}\\
&+\frac{h_{11}}{(\n^{n+1}-a)^2}(F-\Phi)_iu_i-\frac{F^{ii}h_{11ii}}{\n^{n+1}-a}+\frac{h_{11}}{(\n^{n+1}-a)^2}F^{ii}\n^{n+1}_{ii}\\
&=\frac{1}{\n^{n+1}-a}F^{ii}(h^2_{ii}h_{11}-h^2_{11}h_{ii})+\frac{F^{ij, rs}h_{ij1}h_{rs1}}{\n^{n+1}-a}\\
&-\frac{\nabla_{11}\Phi}{\n^{n+1}-a}+\frac{(F-\Phi)\kp_1^2}{\n^{n+1}-a}+\frac{h_{11}}{(\n^{n+1}-a)^2}(F-\Phi)_iu_i\\
&+\frac{h_{11}}{(\n^{n+1}-a)^2}F^{ii}\lt(-\nabla_kh_{ii}u_k-h^2_{ii}\n^{n+1}\rt)\\
&\leq\frac{-ah_{11}}{(\n^{n+1}-a)^2}f_i\kp_i^2-\frac{\Phi\kp_1^2}{\n^{n+1}-a}+\frac{F^{ij, rs}h_{ij1}h_{rs1}}{\n^{n+1}-a}\\
&+\frac{C}{\n^{n+1}-a}-\frac{\Phi_z\kp_1\n^{n+1}}{\n^{n+1}-a}-\frac{\kp_1}{(\n^{n+1}-a)^2}(\Phi_i+\Phi_zu_i)u_i,\\
\end{aligned}
\ee
which yields,
\be\label{ce1.7}
0\leq\frac{-a\kp_1}{(\n^{n+1}-a)^2}f_i\kp_i^2-\frac{\lt(\inf\limits_{\bar{\Omega}\times[-C_0, C_0]}\Phi\rt)\kp_1^2}{\n^{n+1}-a}+C\kp_1,
\ee
thus
\be\label{ce1.8}
{\kappa_1}\leq C=C(\Phi, |u|_{C^1}).\ee
Therefore we conclude that 
\be\label{ce1.9}
\max\limits_{\bar{\Omega}_T}\frac{\kp_{\text{max}}}{\n^{n+1}-a}\leq C_2\lt(1+\max\limits_{\partial\Omega_T}\kp_{\text{max}}\rt).
\ee
\end{proof}

\subsection{$C^2$ boundary estimates}
\label{ce2}
\setcounter{equation}{0}
We use $\nu$ for the outer unit normal of $\partial\Omega$ and $\tau$ for a direction that tangential to $\partial\Omega.$ 
By the exactly same argument as Lemma 4.1 of \cite{SS} we have
\begin{lemma}[Mixed $C^2$ estimates at the boundary]
\label{ce2lm.1'}
Let $u$ be the solution of our flow equation \eqref{int.1}. Then the absolute value  of $u_{\tau\nu}$ remains a priori bounded on 
$\partial\Omega$ during the evolution.
\end{lemma}
Now we consider the function
\be\label{ce2.1}
V(x, \xi, t):=u_{\xi\xi}-2(\xi\cdot\nu)\xi_i'(D_i\vp-D_kuD_i\nu^k),
\ee
where $\xi'=\xi-(\xi\cdot\nu)\nu.$
By Theorem \ref{ce1th.1}, we may assume $V(x, \xi, t)$ achieves its maximum at $(x_0, t_0)\in\partial\Omega\times(0, T],$ otherwise, we are done.

We will devide it into 3 cases.

(i). \textbf{ $\xi$ is tangential.} Computing the second tangential derivatives of the boundary condition we obtain
\be\label{ce2.2}
D_ku\delta_i\delta_j\nu^k+\delta_i\nu^k\delta_jD_ku+\delta_j\nu^k\delta_iD_ku+\nu^k\delta_i\delta_jD_ku=\delta_i\delta_j\vp,
\ee
where $\delta_i=(\delta_{ij}-\nu^i\nu^j)D_i.$
Therefore at $(x_0, t_0)$ we have
\be\label{ce2.3}
\begin{aligned}
D_{\xi\xi\nu}u&=\nu^k\xi_i\xi_jD_{ijk}u\\
&\leq -2(\delta_i\nu^k)D_{jk}u\xi_i\xi_j+(\delta_i\nu^j)\xi_i\xi_jD_{\nu\nu}u+\vp_zD_{ij}u\xi_i\xi_j+C.
\end{aligned}
\ee

Next since $V$ attains its maximum at $(x_0, t_0)$ we have
\be\label{ce2.4}
0\leq D_\nu V=u_{\xi\xi\nu}-a_kD_{k\nu}u-(D_\nu a_k)D_ku-D_\nu b,
\ee
where $a_k=2(\xi\cdot\nu)(\vp_z\xi'_k-\xi'_iD_i\nu^k)$ and $b=2(\xi\cdot\nu)\xi'_k\vp_{x_k}.$
Thus, using Lemma \ref{ce2lm.1'}
\be\label{ce2.5}
u_{\xi\xi\nu}\geq a_\nu D_{\nu\nu}u-C=-C,
\ee
combine with \eqref{ce2.3} yields
\be\label{ce2.6}
-2(\delta_i\nu^k)D_{jk}u\xi_i\xi_j+(\delta_i\nu^j)\xi_i\xi_ju_{\nu\nu}-c_{\varphi}D_{ij}u\xi_i\xi_j+C\geq-C.
\ee
Therefore we have
\be\label{ce2.7}
D_{\xi\xi}u(x_0, t_0)\leq C(1+D_{\nu\nu}u(x_0, t_0)).
\ee

(ii)\textbf{$\xi$ is non-tangential.} We write $\xi=\alpha\tau+\beta\nu,$ where $\alpha=\xi\cdot\tau,$ $\beta=\xi\cdot\nu\neq 0.$
Then
\be\label{ce2.8}
\begin{aligned}
D_{\xi\xi}u&=\alpha^2D^2_{\tau\tau}u+\beta^2D_{\nu\nu}u+2\alpha\beta D_{\tau\nu}u\\
&=\alpha^2D_{\tau\tau}u+\beta^2D_{\nu\nu}u+V'(x, \xi),\\
\end{aligned}
\ee
where $V'=2(\xi\cdot\nu)\xi'_i(D_i\vp-D_kuD_i\nu^k).$
Thus we get,
\be\label{ce2.9}
\begin{aligned}
V(x_0, \xi, t_0)&=\alpha^2 V(x_0, \tau, t_0)+\beta^2 V(x_0, \nu, t_0)\\
&\leq\alpha^2 V(x_0, \xi, t_0)+\beta^2 V(x_0, \nu, t_0),\\
\end{aligned}
\ee
which yeilds
\be\label{ce2.10}
u_{\xi\xi}(x_0, t_0)\leq C(1+u_{\nu\nu}(x_0, t_0)).
\ee

(iii)\textbf{Double normal $C^2$-estimates at the boundary.} Let's recall our evolution equation
\be\label{ce2.11}
\lt\{\begin{aligned}
\dot{u}&=w\lt[F\lt(\frac{\gamma^{ik}u_{kl}\gamma^{lj}}{w}\rt)-\Phi(x, u)\rt]\\
u_\nu&=\vp(x, u)
\end{aligned}\rt.
\ee
 In the following we denote
\[G(D^2u, Du)=F\lt(\frac{\gamma^{ik}u_{kl}\gamma^{lj}}{w}\rt),\] then we have
\be\label{ce2.15}
G^{ij}:=\frac{\partial G}{\partial u_{ij}}=\frac{1}{w}F^{kl}\gamma^{ik}\gamma^{lj},
\ee
\be\label{ce2.16}
G^s:=\frac{\partial G}{\partial u_s}=-\frac{u_s}{w^2}F-\frac{2}{w(1+w)}F^{ij}a_{ik}(wu_k\gamma^{sj}+u_j\gamma^{ks}).
\ee
By the positivity of $[a_{ij}],$ it's easy to see that
\be\label{ce2.17}
\sum|G^i|\leq CF\leq \tilde{C}_0.
\ee

Now, let $q(x)=-d(x)+Nd^2(x),$ then $q\in C^\infty$ in $\Omega_\mu$ for some constant $\mu\leq\tilde{\mu}$ small depending on $\Omega,$
and $N\mu\leq\frac{1}{8}.$
Since 
\[-Dd(y_0)=\nu(x_0)\]
where $x_0\in\partial\Omega$ and $\text{dist}(y_0, \partial\Omega)=\text{dist}(x_0, y_0),$
$q$ satisfies the following properties in $\Omega_\mu:$
\be\label{ce2.18}
-\mu+N\mu^2\leq q\leq 0;\,\, \frac{1}{2}\leq|Dq|\leq 2.
\ee
It's also easy to see that $\frac{Dq}{|Dq|}=\nu$ for unit outer normal $\nu$ on the boundary.

Next,  let
\be\label{ce2.19}
M=\max\limits_{\partial\Omega\times[0, T]}u_{\nu\nu}
\ee
and $Q(x, t)=Q(x)=(A+\frac{1}{2}M)q(x)$ in $\Omega_\mu,$
where $\mu, A, N$ are positive constant to be chosen later.
We consider the following function
\be\label{ce2.20}
P(x, t):=Du\cdot Dq-\vp-Q
\ee
\begin{lemma}
\label{ce2lm.1}
For any $(x, t)\in\bar{\Omega}_\mu\times[0, T],$ if we choose $A, N$ large, $\mu$ small, then we have $P(x, t)\geq 0.$ 
\end{lemma}
\begin{proof}
First, let's assume $P(x, t)$ attains its minimum at $(x_0, t_0)\in\Omega_\mu\times(0, T]$ and $u_{ij}(x_0, t_0)=u_{ii}(x_0, t_0)\delta_{ij}.$
Differentiating $P$ we get
\be\label{ce2.21}
P_i=\sum_lu_{li}q_i+\sum_lu_lq_{li}-\vp_i-Q_i,
\ee
\be\label{ce2.22}
P_{ij}=\sum_lu_{lij}q_l+2\sum_lu_{li}q_{lj}+\sum_lu_lq_{lij}-\vp_{ij}-Q_{ij},
\ee
and
\be\label{ce2.23}
\begin{aligned}
P_t&=Du_t\cdot Dq-\vp-Q\\
&=[w(F-\Phi)]_lq_l-\vp_zu_t=[w(F-\Phi)]_lq_l-\vp_zw(F-\Phi).\\
\end{aligned}
\ee
Therefore at $(x_0, t_0)$ we have
\be\label{ce2.24}
\begin{aligned}
&\frac{1}{w}P_t-G^{ij}P_{ij}\\
&=\frac{1}{w}[w(F-\Phi)]_lq_l-\vp_z(F-\Phi)-G^{ij}(\sum_lu_{lij}q_l+2\sum_lu_{li}q_{lj}\\
&+\sum_lu_lq_{lij}-\vp_{ij})+(A+\frac{1}{2}M)G^{ij}q_{ij}\\
&=\frac{1}{w}[w(F-\Phi)]_lq_l-\vp_z(F-\Phi)-G^{ii}\sum_lu_{lii}q_l\\
&-2G^{ii}u_{ii}q_{ii}-G^{ii}u_lq_{lii}+G^{ii}\vp_{ii}+(A+\frac{1}{2}M)G^{ii}q_{ii}.
\end{aligned}
\ee
This implies at $(x_0, t_0)$
\be\label{ce2.25}
\begin{aligned}
0&\geq\frac{1}{w}P_t-G^{ii}P_{ii}\\
&=\frac{(F-\Phi)}{w}\cdot\frac{u_lu_{ll}q_l}{w}+F_lq_l-\Phi_lq_l-\vp_z(F-\Phi)\\
&-G^{ii}\sum_lu_{lii}q_l-2G^{ii}u_{ii}q_{ii}-\sum_lG^{ii}u_lq_{lii}+G^{ii}(\vp_{x_ix_i}+2\vp_{x_iz}u_i+\vp_zu_{ii})\\
&+(A+\frac{1}{2}M)G^{ii}q_{ii}.\\
\end{aligned}
\ee
Since $G(D^2u, Du)=F$ we have
\be\label{ce2.26}
G^{ij}u_{ijl}+G^su_{sl}=F_l,
\ee
which gives us
\be\label{ce2.27}
F_lq_l-G^{ij}u_{ijl}q_l=G^su_{sl}q_l.
\ee
By \eqref{ce2.17} we have
\be\label{ce2.28}
|G^su_{sl}q_l|=|G^lu_{ll}q_l|\leq\tC_1(M+1).
\ee
Moreover, by the speed estimate \eqref{se.1} and the gradient estimate \eqref{ge.3} it's easy to see
\be\label{ce2.29}
|\Phi_lq_l|+\lt|\frac{F-\Phi}{w}\cdot\frac{u_lu_{ll}q_l}{w}+\vp_zG^{ii}u_{ii}\rt|\leq \tC_2 M.
\ee
Now, by the convexity of $\partial\Omega,$ we can assume
\be\label{ce2.30}
2k_0\delta_{\alpha\beta}\leq-d_{\alpha\beta}\leq k_1\delta_{\alpha\beta},\,\,1\leq\alpha, \beta\leq n-1.
\ee
Thus in $\Omega_\mu$ we have
\be\label{ce2.31}
(k_1+2N)\delta_{ij}\geq q_{ij}=-d_{ij}+2Ndd_{ij}+2Nd_id_j\geq k_0\delta_{ij},
\ee
where $1\leq i, j\leq n.$
We get
\be\label{ce2.32}
|2G^{ii}u_{ii}q_{ii}|\leq\tC_3(k_1+2N).
\ee
Since 
\be\label{ce2.33}
q_{ijl}=-d_{ijl}+2Nd_ld_{ij}+2Ndd_{ijl}+4Nd_{il}d_j,
\ee
we get
\be\label{ce2.34}
|q_{ijl}|\leq C(|\partial\Omega|_{C^3})+6Nk_1.
\ee
Therefore
\be\label{ce2.35}
|G^{ii}u_lq_{lii}|\leq\lt(C(|\partial\Omega|_{C^3})+6Nk_1\rt)C_1\sum G^{ii},\\
\ee
consequently we have
\be\label{ce2.36}
|G^{ii}u_lq_{lii}+G^{ii}(\vp_{x_ix_i}+2\vp_{x_iz}u_i)|\leq (\tC_4+6\tC_5Nk_1)\sum G^{ii}.
\ee
To conclude we obtained
\be\label{ce2.37}
\begin{aligned}
0&\geq\frac{1}{w}P_t-G^{ii}P_{ii}\\
&\geq-\tC_2M-\tC_1(M+1)-\tC_3(k_1+2N)-(\tC_4+6\tC_5 Nk_1)\sum G^{ii}\\
&+(A/2+1/4M)k_0\sum G^{ii}+(A/2+1/4M)G(D^2 q, Du),
\end{aligned}
\ee
here we used the concavity of $f,$ which gives us $G^{ij}(D^2u,Du)q_{ij}\geq G(D^2 q, Du).$
By Lemma 2.2 of \cite{GS}, we may choose $N$ sufficiently large such that
\be\label{ce2.38}
\frac{1}{4}G(D^2 q, Du)\geq 2\tC_1+\tC_2,
\ee
then we choose $A$ such that
\be\label{ce2.39}
\frac{k_0}{2}A>\tC_3(k_1+2N)+\tC_4+6N\tC_5k_1.
\ee
Substitute \eqref{ce2.38} and \eqref{ce2.39} to \eqref{ce2.37} we get
\be\label{ce2.40}
\frac{1}{w}P_t-G^{ij}P_{ij}>0
\ee
at $(x_0, t_0),$ leads to a contradiction.

Finally, since for any $(x, t)\in\partial\Omega\cap\Omega_\mu\times[0, T]$ we have
\[P(x, t)=0.\]
For $(x, t)\in\partial\Omega_\mu\setminus\partial\Omega\times[0, T]$ we have
\[P(x, t)\geq-\tC_6+(A+\frac{1}{2}M)\cdot \frac{1}{2}\mu>0,\]
when $A\geq \frac{2\tC_6}{\mu}.$
Moreover, when $A\geq \tC_7=\tC_7(|u_0|_{C^2}, |\vp|_{C^1}),$ we have for $x\in \Omega_\mu$
\[P(x, 0)\geq 0.\]
Thus, choose 
\[A=\frac{2[\tC_3(k_1+2N)+\tC_4+6N\tC_5k_1]}{k_0}+\frac{2\tC_6}{\mu}+\tC_7\]
we have $P(x, t)\geq 0$ in $\Omega_\mu\times [0, T].$
\end{proof}

\begin{theorem}
\label{ce2th.1}
Let $\Omega$ be a smooth bounded, strictly convex domain in $\R^n,$ u is a smooth solution of \eqref{int.1}, $\nu$ is the outer unit normal vector of 
$\partial\Omega.$ Then we have
\be\label{ce2.41}
\max\limits_{\partial\Omega\times[0, T]}u_{\nu\nu}\leq C.
\ee
\end{theorem}
\begin{proof}
Assume $(z_0, t_0)\in\partial\Omega\times[0, T]$ is the maximum point of $u_{\nu\nu}$ on $\partial\Omega\times[0, T].$ By Lemma \ref{ce2lm.1} we have
\be\label{ce2.42}
\begin{aligned}
0&\geq P_\nu(z_0, t_0)=(\sum_lu_{l\nu}q_l+u_lq_{l\nu}-\vp_\nu)-(A+\frac{1}{2}M)q_\nu\\
&\geq u_{\nu\nu}-C(|u|_{C^1}, N, |\partial\Omega|_{C^2}, |\vp|_{C^1})-(A+\frac{1}{2}M),\\
\end{aligned}
\ee
Therefore we have,
\be\label{ce2.43}
\max\limits_{\partial\Omega\times[0, T]}u_{\nu\nu}\leq C+\frac{1}{2}M,
\ee
which implies \eqref{ce2.41}.
\end{proof}
\bigskip

\section{Convergence to a stationary solution}
\label{cvg}
\setcounter{equation}{0}
Let us go back to our original problem \eqref{int.1}, which is a scalar parabolic differential equation defined on the cylinder
$\Omega_T=\Omega\times[0, T]$ with initial value $u_0.$ In view of a priori estimates, which we have estimated in the preceding sections, we know that
\be\label{cvg.1}
|D^2 u|\leq C,
\ee
\be\label{cvg.2}
|Du|\leq C,
\ee
and 
\be\label{cvg.3}
|u|\leq C.
\ee
Therefore, 
\[\mbox{$F$ is uniformly elliptic.}\]
Moreover, since $F$ is concave, we have uniform $C^{2+\alpha}(\Omega)$ estimates for $u(\cdot, t), \forall t\in[0, T].$ We can repeat the process and conclude that the flow exists for all $t\in[0, \infty).$

By integrating the flow equation with respect to $t$ we get
\be\label{cvg.4}
u(x, t^*)-u(x, 0)=\int_0^{t^*}w(F-\Phi)dt.
\ee
In particular, by \eqref{cvg.3} we have
\be\label{cvg.5}
\int_0^{\infty}w(F-\Phi)dt<\infty\,\,\forall x\in\Omega.
\ee
Hence for any $x\in\Omega$ there existes a sequence $t_k\goto\infty$ such that $F-\Phi\goto 0.$
On the other hand, $u(x, \cdot)$ is monotone increasing and bounded. Therefore,
\be\label{cvg.6}
\lim\limits_{t\goto\infty}u(x, t)=u^\infty(x)
\ee
exists, and is of class $C^\infty(\bar{\Omega})$. Moreover, $u^\infty$ is a stationary solution of our problem, i.e., 
$f(\kp[\Sigma^\infty])=\Phi(x, u^\infty)$ and $u^\infty_\nu=\phi(x, \infty).$

\end{document}